\date{15 August 2014}

\documentclass[11pt,reqno]{amsart}
\usepackage{latexsym,amsmath,amsfonts,amscd,amssymb}
\usepackage{graphics}

\newcommand{\la}{\langle}
\newcommand{\ra}{\rangle}
\newcommand{\x}{\times}
\newcommand{\ox}{\otimes}
\newcommand{\too}{\longrightarrow}

\newcommand{\SD}{{\mathcal{D}}}

\newcommand{\SM}{{\mathcal{M}}}
\newcommand{\SK}{{\mathcal{K}}}
\newcommand{\SF}{{\mathcal{F}}}
\newcommand{\ZZ}{\mathbb{Z}}

\newcommand{\CP}{\mathbb{C}P}
\newcommand{\RR}{\mathbb{R}}

\newcommand{\SU}{\operatorname{SU}}

\DeclareMathOperator{\coker}{coker}
\DeclareMathOperator{\Id}{Id}

\DeclareMathOperator{\Sym}{Sym}
\DeclareMathOperator{\Ant}{Ant}

\newtheorem{theorem}{Theorem}
\newtheorem{proposition}[theorem]{Proposition}

\newtheorem{definition}[theorem]{Definition}

\newtheorem{corollary}[theorem]{Corollary}

\title{Simply connected K-contact and Sasakian manifolds of dimension $7$}

\author[V. Mu\~{n}oz]{Vicente Mu\~{n}oz}
\address{Facultad de Ciencias Matem\'aticas, Universidad
Complutense de Madrid, Plaza de Ciencias
3, 28040 Madrid, Spain}
\address{Instituto de Ciencias Matem\'aticas (CSIC-UAM-UC3M-UCM),
C/ Nicol\'as Cabrera 15, 28049 Madrid, Spain}
\email{vicente.munoz@mat.ucm.es}

\author[A. Tralle]{Aleksy Tralle}
\address{Department of Mathematics and Computer Science, University of Warmia
and Mazury, S\l\/oneczna 54, 10-710, Olsztyn, Poland}
\email{tralle@matman.uwm.edu.pl}



\begin{document}

\begin{abstract}
 We construct a compact simply connected $7$-dimensional manifold admitting a K-contact
structure but not a Sasakian structure. We also study rational homotopy properties of
such manifolds, proving in particular that a simply connected $7$-dimensional Sasakian manifold
has vanishing cup product $H^2 \x H^2 \to H^4$ and that it is formal if and only if all its
triple Massey products vanish.
\end{abstract}

\maketitle

\section{Introduction}\label{sec:1}

Sasakian geometry has become an important and active subject, especially after the appearance of
the fundamental treatise of Boyer and Galicki \cite{BG}. Chapter 7 of this
book contains an extended discussion of the topological problems in the theory of Sasakian, and, more generally,
K-contact manifolds. These are odd-dimensional analogues to K\"ahler and symplectic manifolds, respectively. 

The precise definition is as follows.
Let $(M,\eta)$ be a co-oriented contact manifold with a contact form
$\eta\in \Omega^1(M)$, that is $\eta\wedge (d\eta)^n>0$ everywhere, with $\dim M=2n+1$. 
We say that $(M,\eta)$ is {\em K-contact} if there is an
endomorphism $\Phi$ of $TM$ such that:
 \begin{itemize}
\item $\Phi^2=-\Id + \xi\otimes\eta$, where $\xi$ is the Reeb
vector field of $\eta$ (that is $i_\xi \eta=1$, $i_\xi (d\eta)=0$),
\item the contact form $\eta$ is compatible with $\Phi$ in the sense that
$d\eta (\Phi X,\Phi Y)\,=\,d\eta (X,Y)$, for all vector fields $X,Y$,
\item $d\eta (\Phi X,X)>0$ for all nonzero $X\in \ker \eta$, and
\item the Reeb field $\xi$ is Killing with respect to the
Riemannian metric defined by the formula
 $g(X,Y)\,=\,d\eta (\Phi X,Y)+\eta (X)\eta(Y)$.
\end{itemize}
In other words, the endomorphism $\Phi$ defines a
complex structure on $\SD=\ker \eta$ compatible with $d\eta$, hence
$\Phi$ is orthogonal with respect to the metric
$g|_\SD$. By definition, the Reeb vector
field $\xi$ is orthogonal to $\ker \eta$, and it is a 
Killing vector field.

Let $(M,\eta,g,\Phi)$ be a K-contact manifold. Consider the contact cone as the Riemannian manifold
$C(M)=(M\times\mathbb{R}^{>0},t^2g+dt^2)$.
One defines the almost complex structure $I$ on $C(M)$ by:
 \begin{itemize}
\item $I(X)=\Phi(X)$ on $\ker\eta$,
\item $I(\xi)=t{\partial\over\partial t},\,I(t{\partial\over\partial t})=-\xi$, for the Killing vector field $\xi$ of $\eta$.
\end{itemize}
We say that $(M,\eta,\Phi,g,I)$ is  {\it Sasakian} if $I$ is integrable.
Thus, by definition, any Sasakian manifold is K-contact.

There are several topological obstructions to the existence of the aforementioned structures 
on a compact manifold $M$ of dimension $2n+1$, for example:
 \begin{enumerate}
\item the evenness of the $p$-{th} Betti number for $p$ odd with $1\, \leq\, p \, \leq\, n$, of a Sasakian manifold,
\item some torsion obstructions in dimension $5$ discovered by Koll\'ar \cite{Kollar},
\item the fundamental group of Sasakian manifolds are special,
\item the cohomology algebra of a Sasakian manifold satisfies the hard Lefschetz property,
\item formality properties of the rational homotopy type.
\end{enumerate}

An early result \cite{Betti} establishes that the odd Betti numbers up to the middle dimension of Sasakian manifolds 
must be even. The parity of $b_1$ was used to produce the first examples of K-contact manifolds with no Sasakian structure
\cite[example 7.4.16]{BG}. More refined tools are needed in the case of even Betti numbers. The 
cohomology algebra of a Sasakian manifold satisfies a hard Lefschetz property \cite{CNY}. Using
it examples of K-contact non-Sasakian manifolds are produced in \cite{CNMY} in dimensions $5$ and $7$. These
examples are nilmanifolds with even Betti numbers, so in particular they are not simply connected.

The fundamental group can also be used to construct K-contact non-Sasakian manifolds.
Fundamental groups of Sasakian manifolds are called Sasaki groups, and satisfy
strong restrictions. Using this it is possible to construct (non-simply connected) compact 
manifolds which are K-contact but not Sasakian \cite{Chen}. 

When one moves to the case of simply connected manifolds, K-contact non-Sasakian 
examples of any dimension $\geq 9$ were constructed in \cite{HT} using the evenness of the third Betti
number of a compact Sasakian manifold. Alternatively, using the hard Lefschetz property
for Sasakian manifolds there are examples \cite{Lin} of simply connected K-contact non-Sasakian
manifolds of any dimension $\geq 9$.

In \cite{Tievsky} and in \cite{BFMT} the rational homotopy type of Sasakian manifolds is studied.
In \cite{BFMT} it is proved that all higher order Massey products for simply connected Sasakian manifolds vanish,
although there are Sasakian manifolds with non-vanishing triple Massey products.
This yields examples of simply connected K-contact non-Sasakian manifolds in dimensions $\geq 17$. 
However, Massey products are not suitable for the analysis of lower dimensional manifolds.

Hence, the problem of the existence of simply connected K-contact non-Sasakian compact manifolds 
(open problem 7.4.1 in \cite{BG}) is still open in dimensions $5$ and $7$. 
Dimension $5$ is the most difficult one, and it is treated in \cite{BG} separately.  Here one
has to use the obstructions of \cite{Kollar} which are very subtle torsion obstructions associated to the
classification of K\"ahler surfaces.
By definition, a simply connected compact oriented $5$-manifold is called a {\it Smale-Barden manifold}. 
These manifolds are classified topologically by $H_2(M,\mathbb{Z})$ and the second Stiefel-Whitney class. 
Chapter 10 of the book by Boyer and Galicki is devoted to a description of some Smale-Barden manifolds which 
carry Sasakian structures. The following problem is still open (open problem 10.2.1 in \cite{BG}). 

{\it Do there exist Smale-Barden manifolds which carry K-contact but do not carry Sasakian structures?}

In this note we solve the described problem in the easier case of dimension $7$ 
(the solution is still possible by means of homotopy theory combined with symplectic surgery).

\begin{theorem} \label{thm:main}
There exist $7$-dimensional compact simply connected K-contact manifolds which do not admit a Sasakian structure.
\end{theorem}

We then turn around  to the study of the rational homotopy type of K-contact and Sasakian simply connected manifolds
of dimension $7$. In particular, we prove:

\begin{corollary} \label{cor:2}
 Let $M$ be a simply connected compact K-contact $7$-dimensional manifold. Suppose
that the cup product map $H^2(M)\x H^2(M)\too H^4(M)$ is non-zero. Then $M$ does
not admit a Sasakian structure.
\end{corollary}

Formality is a very useful rational homotopy property that has been widely used
to distinguish between symplectic and K\"ahler manifolds \cite{TO} (see Section \ref{sec:4} for
definitions and details). Simply connected compact manifolds of dimension $\leq 6$ are 
always formal, so formality becomes interesting in dimension $7$. We study this
property in detail giving a precise characterisation for Sasakian manifolds (see
Theorem \ref{thm:formal-7-dim}). In particular, we have the following:

\begin{corollary} \label{cor:formal-7-dim}
   Let $M$ be a simply connected compact Sasakian $7$-dimensional manifold.
 Then $M$ is formal if and only if all triple Massey products are zero.
\end{corollary}

\medskip

\noindent\textbf{Acknowledgements.} We thank Johannes  Nordstr\"om for several illuminating
conversations and for providing us with a copy of \cite{Johannes}. We are also
grateful to G. Bazzoni and B. Cappelletti-Montano for pointing us some references. The first
author was partially supported by Project MICINN (Spain) MTM2010-17389. 
The second author was partly supported by the ESF Research Networking 
Programme CAST (Contact and Symplectic Topology).

\section{Gompf-Cavalcanti manifold}\label{sec:2}

Let $(M,\omega)$ be a symplectic manifold of dimension $2n$. For every $0\leq k \leq n$,
we define the Lefschetz
map as $L_\omega:H^{n-k}(M) \to H^{n+k}(M)$, $L_\omega([\beta])=[\beta\wedge\omega^{n-k}]$.
We say that $M$ satisfies the hard Lefschetz property if $L_\omega$ is an isomorphism for
every $0 \leq k \leq n$.

\begin{proposition} \label{prop:GC}
There exists a simply connected $6$-dimensional symplectic manifold $(M,\omega)$ such that
$\dim \ker (L_\omega: H^2(M)\rightarrow H^4(M))$ is odd.
\end{proposition}

\begin{proof}
Gompf  constructs in \cite[Theorem 7.1]{Gompf}  an example of a simply connected $6$-dimensional symplectic manifold $(M,\omega)$
which does not satisfy the hard Lefschetz property, that is,
the Lefschetz map $L_\omega: H^2(M)\rightarrow H^4(M)$ is not an isomorphism. 
If $\dim \ker L_\omega$ is already odd then we have finished. 

So let us suppose that $\dim \ker L_\omega$ is even. Take  a 
cohomology class $a\in H^2(M)$ which belongs to the kernel of   $L_\omega$.
In \cite[Lemma 2.4]{C} Cavalcanti proves that given a symplectic manifold $(M,\omega)$ as above satisfying that there
exists a symplectic surface $S\hookrightarrow M$ with $\la a,[S]\ra \neq 0$, then there is another $6$-dimensional 
symplectic manifold $(M',\omega')$ (the symplectic blow-up of $M$ along $S$) satisfying
 $$
 \dim \ker (L_{\omega'}: H^2(M')\rightarrow H^2(M'))= \dim \ker (L_\omega: H^2(M)\rightarrow H^2(M))-1.
 $$ 
The symplectic blow-up of $M$ along $S$ is constructed in \cite{McDuff}, where it is proved that 
the fundamental groups $\pi_1(M')\cong \pi_1(M)$,
hence $M'$ is simply connected. 
This means that the simply connected $6$-dimensional symplectic manifold $M'$ satisfies that
$\dim \ker (L_{\omega'} : H^2(M')\rightarrow H^4(M'))$ is odd, as required.

It remains to find $S\hookrightarrow M$ as required. The cohomology class $a$ is non-zero, so there
is some $b\in H^4(M,\ZZ)$ such that $a\cup b\neq 0$. It is easy to see that there is a
rank $2$ complex vector bundle $E\to M$ with $c_1(E)=0$, $c_2(E)=2b$. This corresponds
to the fact that the map $[M,B\SU(2)]\to H^4(M,\ZZ)$ given by the second Chern class exhausts
$2\,  H^4(M,\ZZ)$. A short proof runs as follows: $B\SU(2)$ has trivial $3$-skeleton and 
it has $\pi_4(B\SU(2))=\ZZ$ and $\pi_5(B\SU(2))=\ZZ_2$. Represent
the cohomology class $b$ by a cocycle $\varphi_b:C_4(M)\to \ZZ$, where $C_4(M)$ is the space of 
cellular chains. Given $b$, we define $f:M\to B\SU(2)$ inductively on the skeleta (in what follows we denote by $X[k]$ the $k$-skeleton of a  space $X$).  It is
trivial on the $3$-skeleton of $M$. For every $4$-cell $c$, we define $f:c\to B\SU(2)[4]=S^4$
to have degree $\varphi_b(c)\in \ZZ$. As $M$ is simply connected there are no $5$-cells, so
it only remains to attach the $6$-cell $c_6$ to the $4$-skeleton $M[4]$. The 
attaching map is given by some $g:S^5\to M[4]$. When composed with $f$, we have
a map $f\circ g: S^5 \to  B\SU(2)$, which gives an {\it obstruction} element $o_f\in \pi_5(B\SU(2))=\ZZ_2$.
If we multiply $b$ by two, then the map $\varphi_b$ gets multiplied by $2$. The
corresponding $f$ is given by composing $f$ with a double cover of $S^4$, hence
the obstruction element is $2o_f=0$. This means that the map $f$ associated to $2b$ 
can be extended to $M\to B\SU(2)$.

Now take the rank $2$ bundle $E \to M$ just constructed. Assume that $[\omega]$ is a
an integral cohomology class (which can always be done by perturbing $\omega$ slightly
to make it rational and multiplying it by a large integer). Let $L\to M$ be the
line bundle with first Chern class $c_1(L)=[\omega]$. We now use
the asymptotically holomorphic techniques introduced by Donaldson in \cite{Donaldson}.
Specifically, the result of \cite{Auroux} guarantees the existence of a suitable
large $k\gg 0$ and a section of $E\ox L^{\ox k}$ whose zero locus is a symplectic
manifold (an asymptotically holomorphic manifold in fact). This zero locus $S\subset M$
is a symplectic surface, and the cohomology class defined by $S$ is $c_2(E\otimes L^{\ox k})=
c_2(E)+2k c_1(L)=2b+2k[\omega]$.
Therefore $\la a,[S]\ra =\la a, 2b+2k[\omega] \ra=2\la a,b\ra \neq 0$, as required.
\end{proof}

We will call the manifold produced in Proposition \ref{prop:GC} the {\it Gompf-Cavalcanti manifold},
because it is constructed by the surgery technique of Gompf \cite{Gompf} together with the
symplectic blow-up of Cavalcanti \cite{C}. Note however that this is not a unique one but a 
family of manifolds.

\section{Simply-connected K-contact non-Sasakian manifolds in dimension $7$}\label{sec:3}

We show the existence of simply connected compact K-contact non-Sasakian manifolds in 
dimension $7$ by proving that the Boothby-Wang 
fibration over the Gompf-Cavalcanti manifold is K-contact but non-Sasakian. 
The existence of a K-contact structure on such fibration is 
shown in \cite{BFMT} and \cite{HT}. For the convenience of the reader we briefly recall these constructions.

Let $(B\, ,\omega)$ be a symplectic manifold such that the
cohomology class $[\omega]$ is integral. Consider the principal $S^1$-bundle
$\pi : M\to  B$ given by the cohomology class $[\omega]\,\in\, H^2(B,\,\mathbb{Z})$.
Fibrations of this kind were first considered  by Boothby and Wang  and are
called {\em Boothby-Wang fibrations}. By \cite{W}, the total space $M$
carries an $S^1$-invariant contact form $\eta$ such that $\eta$
is a connection form whose curvature is $d\eta=\pi^*\omega$.  We
have the following result.

\begin{theorem}\label{thm:Boothby-Wang}
Any Boothby-Wang fibration admits a K-contact structure on the total space.
\end{theorem}

\begin{proof}
To prove this theorem we need to introduce a certain tool, called the universal contact moment map in the sense of Lerman \cite{L}. 
Recall that by our assumption the given contact distribution $\SD$ is determined by the contact form $\eta$, that is $\SD=\ker\eta$. 
Consider its annihilator $\SD^0\subset T^*M$. Clearly, $\SD^0$ is a line bundle, and, therefore, it has two components after the removal of the zero section,
 $$
 \SD^0\setminus M=\SD^0_+\sqcup \SD^0_{-}.
 $$
Single out one of these components, say $\SD^0_{+}$. Consider the Lie algebra of contact vector fields $\chi(M,\eta)$ on $M$. 
It is known that this Lie algebra can be identified with a space of sections of the vector bundle $TM/\SD$, 
that is $\chi(M,\eta)\cong\Gamma(M,TM/\SD)$. Because of that there is a 
natural pairing between points of the line bundle $\SD^0$ and contact vector fields given by the formula
 $$
 \SD^0\times\chi(M,\eta)\rightarrow \mathbb{R},\quad ((p,\beta),X)\mapsto \langle\beta,X_p\rangle
 $$
where $\beta\in\SD^0,X_p\in T_pM,  p\in M$. Suppose that a Lie algebra $\mathfrak{g}$ acts on $M$ by contact vector fields, that is, 
there exists a representation $\rho:\mathfrak{g}\rightarrow\chi(M,\eta)$. Define the {\it universal moment map} as the map
 $$
 \psi:\SD_{+}^0\rightarrow \mathfrak{g}^*
$$
by the formula
 $$
 \langle\psi (p,\beta),X\rangle=\langle(p,\beta),\rho(X)\rangle=\langle\beta,\rho(X)_p\rangle,
 $$
where $(p,\beta)\in(\SD_{+}^0)_p\subset T^*_pM,X\in\mathfrak{g}$. Now the proof becomes a consequence  of the following criterion proved by Lerman in \cite{L}.

\begin{proposition}\label{prop:L-K}
A compact co-orientable contact manifold $(M,\eta)$ admits a K-contact metric $g$ if and only if there exists an action of a torus $T$ on $M$ 
preserving the contact structure $\SD$ and a vector $X\in\mathfrak{t}=L(T)$ so that the function
$\langle\psi,X\rangle:\SD^0_{+}\rightarrow \mathbb{R}$
is strictly positive. \hfill $\Box$
\end{proposition}

We continue with the proof of Theorem \ref{thm:Boothby-Wang}.
Consider the $S^1$-action on $M$ given by the Reeb vector field. Let $\mathfrak{g}=L(S^1)$, and $\rho:\mathfrak{g}\rightarrow \chi(M,\eta)$ 
be the homomorphism of Lie algebras determined by this action (thus, $\mathfrak{g}=\mathfrak{t}=L(S^1)$ in this particular situation). 
Since the $S^1$-action is free, $\rho(X)_p\not=0$ for any $p\in M$. Now,
 $$
 \langle\psi,X\rangle(p,\beta)=\langle\psi(p,\beta),X\rangle=\langle\beta,\rho(X)_p\rangle.
 $$
Note that in the considered case $\beta\in (\SD^0_{+})_p\subset T_p^*M$, and, therefore, $\beta\not=0$. Also $(p,\beta)$ belongs to 
the annihilator of the distribution $\SD$, while $\rho(X)$ is transversal to $\SD$, since it is given by the Reeb vector field. Thus, for any point $p$, 
$\langle (p,\beta),\rho(X)_p\rangle\not=0$. Hence, $X$ may be chosen to yield positive sign everywhere, and we complete the proof by 
applying Proposition \ref{prop:L-K}.
\end{proof} 

The following gives a proof of Theorem \ref{thm:main}.

\begin{theorem} 
The total space of the Boothby-Wang fibration over the Gompf-Cavalcanti manifold is 
a simply connected K-contact non-Sasakian manifold of dimension $7$.
\end{theorem}

\begin{proof}
Let $(M,\omega)$ be a Gompf-Cavalcanti manifold as given by Proposition \ref{prop:GC}.
We can assume that $[\omega]$ is an integral cohomology class. Let
 \begin{equation}\label{eqn:M}
  S^1\rightarrow E\rightarrow M
 \end{equation}
be the associated Boothby-Wang fibration. By Theorem \ref{thm:Boothby-Wang}, $E$ has
a  K-contact structure. Now we need to prove that $E$ cannot carry Sasakian structures.

There is an exact sequence 
 $$
 H_2(M)\to H_1(S^1)={\mathbb Z}\to H_1(E) \to 0
 $$ 
from the Serre spectral sequence. The map $H_2(M)\to {\mathbb Z}$ is cupping with $[\omega]\in H^2(M)$. Taking $[\omega]$ integral
cohomology class and primitive, we have that $H_2(M)\to {\mathbb Z}$ is surjective and
hence $H_1(E)=0$. The long homotopy 
exact sequence gives $\pi_1(S^1)={\mathbb Z} \to \pi_1(E) \to \pi_1(M)=0$, hence $\pi_1(E)$ is abelian.
Therefore $E$ is simply connected.

The Gysin exact sequence associated to (\ref{eqn:M}) is
 $$
   H^1(M)=0 \stackrel{\wedge \omega}{\longrightarrow} H^3(M) \longrightarrow  H^3(E) \longrightarrow  H^2(M) \stackrel{\wedge \omega}{\longrightarrow} H^4(M).
 $$
Thus 
 $$
 b^3(E)=b^3(M)+ \dim (\ker L_\omega :H^2(M)  \to H^4(M)).
 $$

As $M$ is a $6$-manifold, we have that $b^3(M)$ is even (by Poincar\'e duality, the intersection
pairing on $H^3(M)$ is an antisymmetric non-degenerate bilinear form, hence the dimension of
$H^3(M)$ is even). By construction,
$\dim (\ker L_\omega :H^2(M) \to H^4(M))$ is odd, so $b^3(E)$ is odd.
As the third Betti number of a $7$-dimensional Sasakian manifold has to be even \cite{Betti}, 
we have that $E$ cannot admit a Sasakian structure.
\end{proof}

\section{Regularity and quasi-regularity}

A Sasakian or a K-contact structure on a compact manifold 
$M$ is called \textit{quasi-regular} if there is a positive
integer $\delta$ satisfying the condition
that each point of $M$ has a foliated coordinate chart
$(U\, ,t)$ with respect to $\xi$ (the coordinate $t$ is in the direction of $\xi$)
such that each leaf for $\xi$ passes through $U$ at most $\delta$ times. If
$\delta\,=\, 1$, then the Sasakian or K-contact structure is called \textit{regular}. (See
\cite[p. 188]{BG}.)

If $N$ is a K\"ahler manifold whose K\"ahler form $\omega$ 
defines an integral cohomology class, then the total space of the circle bundle 
$S^1 \hookrightarrow M \stackrel{\pi}{\longrightarrow} N$ with Euler class $[\omega]\in 
H^2(M,\mathbb{Z})$ is a regular Sasakian manifold with contact form $\eta$ such that $d 
\eta \,=\, \pi^*(\omega)$. The converse also holds: if $M$ is a regular Sasakian structure 
then the space of leaves $N$ is a K\"ahler manifold, and we have a circle
bundle $S^1\to M \to N$ as above. 
If $M$ has a quasi-regular Sasakian structure, then the space of leaves $N$ is a K\"ahler
orbifold with cyclic quotient singularities, and there is an orbifold circle bundle
$S^1 \to M\to N$ such that the contact form $\eta$ satisfies $d\eta=\pi^*(\omega)$,
where $\omega$ is the orbifold K\"ahler form.

Similar properties hold in the K-contact case, substituting K\"ahler by symplectic (actually almost K\"ahler).
If $M$ has a regular K-contact structure, then it is the total space of a circle bundle
$S^1 \hookrightarrow M \stackrel{\pi}{\longrightarrow} N$, where $(N,\omega)$ is a symplectic
manifold, with Euler class $[\omega]\in 
H^2(M,\mathbb{Z})$ and $d \eta \,=\, \pi^*(\omega)$. If $M$ has a quasi-regular K-contact structure,
then it is the total space of an orbifold circle bundle $S^1 \to M\to N$ over a symplectic
orbifold $N$ with cyclic quotient singularities and Euler class $[\omega]\in 
H^2(M,\mathbb{Z})$, where $\omega$ is the orbifold symplectic form.

A result of \cite{OV} says that if $M$ admits a Sasakian structure,
then it admits also a quasi-regular Sasakian structure. This also extends to the
case of K-contact structures.

\begin{proposition} \label{prop:quasi-regular}
If a compact manifold $M$ admits a K-contact structure, it admits a quasi-regular contact structure.
\end{proposition}

\begin{proof}
Assume that there is a K-contact structure on $M$. By Proposition \ref{prop:L-K}, there exists a torus action 
$T\times M\rightarrow M$ preserving the contact distribution and a vector 
$X\in\mathfrak{t}$ such that $\langle\psi,X\rangle>0$. Choose a vector $Y\in\mathfrak{t}$ with the 
property that it is tangent to an embedding $T'=S^1\hookrightarrow T$. Clearly, 
the corresponding fundamental vector field $Y_M$ has the property that the leaves 
of the corresponding foliations are compact. The set of such $Y$ is dense in $\mathfrak{t}$. 
Therefore, for vectors $Y$ which are sufficiently close to $X$, the condition $\langle\psi, Y\rangle>0$ 
is still satisfied. 

So it remains to see that there is K-contact structure whose Reeb vector field is $Y_M$, since this
will be quasi-regular because the leaves of the characteristic foliation are all compact.
We follow the notations of the proof of Theorem \ref{thm:Boothby-Wang}.
The action of the circle $T'$ on $M$ preserves $\SD$, hence
the lifted action of $T'$ on $T^*M$ preserves $\SD^0$. Since $T'$ is connected, 
the lifted action preserves the connected component $\SD^0_+$ as well. It follows that for any $1$-form $\beta$ on $M$ 
with $\ker \beta=\SD$, the average $\bar\beta$ of $\beta$ over $T'$ still satisfies $\ker\,\bar\beta=\SD$. 
So $\bar\beta \in \SD^0$. Now use the formula (derived in \cite{L}, formulae (3.4) and (3.5)),
 $$
 i_{Y_M}\bar\beta=\langle\psi\circ\bar\beta,Y\rangle >0.
 $$
Now let
 $$
 \eta=(\langle\psi\circ\bar\beta,Y\rangle)^{-1}\bar\beta,
 $$
which satisfies  $i_{Y_M}\eta=1$. Hence $\eta$ defines the contact structure and $Y_M$ is its Killing vector field.
Then $TM=\SD\oplus\la Y_M\ra $, and the splitting is $T'$-invariant. We use the splitting to define the desired 
Riemannian metric $g$. Declare $\SD$ and $\la Y_M\ra$ to be orthogonal and define $g(Y_M,Y_M)=1$, 
thus $Y_M$ becomes a unit normal to $\SD$. On $\SD$ we choose a $T'$-invariant complex structure compatible with 
$d\eta|_{\SD}$ and define $g|_{\SD}(\cdot,\cdot)=d\eta|_{\SD}(\cdot,\Phi\cdot)$. 
Then $g$ is $T'$-invariant and hence $L_{Y_M}g=0$. Thus we have obtained a K-contact structure on $M$. 
\end{proof}

\section{Minimal models and formality} \label{sec:3.5}

Now we want to analyse the rational homotopy type of K-contact and 
Sasakian simply connected $7$-manifolds, in particular the property of formality.
Simply connected compact manifolds of dimension $\leq 6$ are always formal \cite{MathZ},
so dimension $7$ is the first instance in which formality is an issue. 

We start by reviewing concepts about minimal models and formality from \cite{MathZ,  FHT, GM}.
A {\it differential graded algebra} (or DGA) over the real numbers $\RR$, is a pair $(A,d)$ consisting
of a graded commutative algebra $A=\oplus_{k\geq 0} A^k$ over $\RR$, and
a differential $d$ satisfying the Leibnitz rule 
$d(a\cdot b) = (da)\cdot b +(-1)^{|a|} a\cdot (db)$, where
$|a|$ is the degree of $a$. 
Given a differential graded commutative algebra $({A},\,d)$, we 
denote its cohomology by $H^*({A})$. The cohomology of a 
differential graded algebra $H^*({A})$ is naturally a DGA with the 
product inherited from that on ${A}$ and with the differential
being identically zero. The DGA $({A},\,d)$ is {\it connected} if 
$H^0({A})\,=\,\RR$, and ${A}$ is {\em $1$-connected\/} if, in 
addition, $H^1({A})\,=\,0$. Henceforth we shall assume that all our DGAs are connected.
In our context, the main example of DGA is the de Rham complex $(\Omega^*(M),\,d)$
of a connected differentiable manifold $M$, where $d$ is the exterior differential.

Morphisms between DGAs are required to preserve the degree and to commute 
with the differential. A morphism $f:(A,d)\to (B,d)$ is a quasi-isomorphism if
the map induced in cohomology $f^*:H^*(A,d)\to H^*(B,d)$ is an isomorphism.
Quasi-isomorphism produces an equivalence relation in the category of DGAs.

A DGA $(\mathcal{M},\,d)$ is {\it minimal\/} if
\begin{enumerate}
 \item $\mathcal{M}$ is free as an algebra, that is, $\mathcal{M}$ is the free
 algebra $\bigwedge V$ over a graded vector space $V\,=\,\bigoplus_i V^i$, and
 \item there is a collection of generators $\{x_\tau\}_{\tau\in I}$
indexed by some well ordered set $I$, such that
 $|x_\mu|\,\leq\, |x_\tau|$ if $\mu \,< \,\tau$ and each $d
 x_\tau$ is expressed in terms of preceding $x_\mu$, $\mu\,<\,\tau$.
 \end{enumerate}
We say that $(\bigwedge V,\,d)$ is a {\it minimal model} of the
differential graded commutative algebra $({A},\,d)$ if $(\bigwedge V,\,d)$ is minimal and there
exists a quasi-isomorphism $\rho\colon
{(\bigwedge V,\,d)}\longrightarrow {({A},\,d)}$.
A connected DGA $({A},\,d)$ has a minimal model unique up to isomorphism. For $1$-connected
DGAs, this is proved in~\cite{DGMS}. In this case, the minimal model satisfies that
$V^1=0$ and the condition (2) above is equivalent to  $dx_\tau$ not having a linear part.

A {\it minimal model\/} of a connected differentiable manifold $M$
is a minimal model $(\bigwedge V,\,d)$ for the de Rham complex
$(\Omega^*(M),\,d)$ of differential forms on $M$. If $M$ is a simply
connected manifold, then the dual of the real homotopy vector
space $\pi_i(M)\otimes \RR$ is isomorphic to $V^i$ for any $i$ (see~\cite{DGMS}).

A {\it model\/} of a DGA $(A,d)$ is any DGA $(B,d)$ with the same minimal 
model (that is, they are equivalent with respect to the equivalence relation
determined by the quasi-isomorphisms).

A minimal algebra $(\bigwedge V,\,d)$ is called {\it formal} if there exists a
morphism of differential algebras $\psi\colon {(\bigwedge V,\,d)}\,\longrightarrow\,
(H^*(\bigwedge V),0)$ inducing the identity map on cohomology.
Also a differentiable manifold $M$ is called formal if its minimal model is
formal. The formality of a minimal algebra is characterized as follows.

\begin{proposition}[\cite{DGMS}]\label{prop:criterio1}
A minimal algebra $(\bigwedge V,\,d)$ is formal if and only if the space $V$
can be decomposed into a direct sum $V\,=\, C\oplus N$ with $d(C) \,=\, 0$
and $d$ injective on $N$, such that every closed element in the ideal
$I(N)$ in $\bigwedge V$ generated by $N$ is exact.
\end{proposition}

This characterization of formality can be weakened using the concept of
$s$-formality introduced in \cite{MathZ}.

\begin{definition}\label{def:primera}
A minimal algebra $(\bigwedge V,\,d)$ is $s$-formal ($s> 0$) if for each $i\leq s$
the space $V^i$ of generators of degree $i$ decomposes as a direct
sum $V^i=C^i\oplus N^i$, where the spaces $C^i$ and $N^i$ satisfy
the three following conditions:
\begin{enumerate}
\item $d(C^i) = 0$,
\item the differential map $d\colon N^i\longrightarrow \bigwedge V$ is
injective, and
\item any closed element in the ideal
$I_s=I(\bigoplus\limits_{i\leq s} N^i)$, generated by the space
$\bigoplus\limits_{i\leq s} N^i$ in the free algebra $\bigwedge
(\bigoplus\limits_{i\leq s} V^i)$, is exact in $\bigwedge V$.
\end{enumerate}
\end{definition}

A differentiable manifold $M$ is $s$-formal if its minimal model
is $s$-formal. Clearly, if $M$ is formal then $M$ is $s$-formal, for any $s>0$.
The main result of \cite{MathZ} shows that sometimes the weaker
condition of $s$-formality implies formality.

\begin{theorem}[\cite{MathZ}]\label{fm2:criterio2}
Let $M$ be a connected and orientable compact
manifold of dimension $2n$ or $(2n-1)$. Then $M$ is formal if and
only if it is $(n-1)$-formal.
\end{theorem}

By Corollary 3.3 in \cite{MathZ} a simply connected compact manifold 
is always $2$-formal. Therefore, Theorem \ref{fm2:criterio2} implies that any
simply connected compact manifold of dimension not more than six
is formal. For simply connected $7$-dimensional compact manifolds, we have that $M$ is
formal if and only if $M$ is $3$-formal. 

Theorem \ref{fm2:criterio2} also holds for compact connected orientable orbifolds,
since the proof of \cite{MathZ} only uses that the cohomology $H^*(M)$ is a 
Poincar\'e duality algebra.

\section{Homotopy properties of simply connected Sasakian $7$-manifolds} \label{sec:4}

\begin{proposition} \label{prop:model}
Let $M$ be a simply connected compact K-contact $7$-dimensional manifold. Then a model
for $M$ is $(H\otimes \bigwedge(x),d)$, where $H$ is the cohomology algebra of a simply connected
symplectic $6$-dimensional orbifold and $dx=\omega \in H^2$ is the class of the symplectic
form.

If $M$ is Sasakian, then $H$ is the cohomology algebra of a simply connected $6$-dimensional 
K\"ahler orbifold.
\end{proposition}

\begin{proof}
Suppose $M$ admits a Sasakian structure. Then $M$ admits a quasi-regular
Sasakian structure \cite{OV}. Therefore, there is an orbifold circle bundle
 $S^1 \to M \to B$, where $B$ is a compact K\"ahler orbifold of
dimension $6$, with Euler class given by the K\"ahler form $\omega\in H^2(B)$. 
We note that $B$ is simply connected because $M$ is so
(see \cite[Theorem 4.3.18]{BG}). In particular, $S^1 \to M \to B$ is
a rational fibration, hence if $\SM$ is a model for $B$, then 
$\SM \otimes \bigwedge(x)$, with $|x|\,=\,1$, $dx\,=\,\omega$, is a model for $M$.

Now $B$ is a simply connected compact orbifold of dimension $6$.
So it is $2$-formal. Theorem \ref{fm2:criterio2} also holds for orbifolds, hence
$B$ is formal. Therefore $\SM \sim (H,0)$, where $H=H^*(B)$ is the
cohomology algebra of $B$. So a model for $M$ is of the form
$(H\otimes \bigwedge(x),d)$, $dx=\omega \in H^2$.

The case where $M$ admits a K-contact structure is similar.
By Proposition \ref{prop:quasi-regular}, it admits a quasi-regular K-contact
structure. Therefore, $M$ is an orbifold
$S^1$-bundle over a symplectic orbifold $S^1 \to M \to B$, with
Euler class given by the orbifold symplectic form $\omega \in H^2(B)$.
As above, a model for $M$  is
$(H\otimes \bigwedge(x),d)$, $dx=\omega \in H^2$, where
$H=H^*(B)$.
\end{proof}

We prove now Corollary \ref{cor:2}.

\begin{corollary}
 Let $M$ be a simply connected compact K-contact $7$-dimensional manifold. Suppose
that the cup product map $H^2(M)\x H^2(M)\too H^4(M)$ is non-zero. Then $M$ does
not admit a Sasakian structure.
\end{corollary}

\begin{proof}
 Let us compute the cohomology of $M$ from its model $(\SM,d)=(H\otimes \bigwedge(x),d)$,
$dx=\omega$, where $H=H^*(B)$ is the cohomology algebra of
a $6$-dimensional simply connected symplectic manifold. Note that $\omega\in H^2$ is
a non-zero element with $\omega^3\in H^6$ generating the top cohomology. 

Consider the Lefschetz map $L_\omega:H^*\to H^{*+2}$, and let $K^*=\ker L_\omega$, $Q^*=\coker L_\omega$.
We have a (non-canonical) isomorphism  $H^i(M) \cong Q^i \oplus K^{i-1} x$. Note that
$Q^3=K^3=H^3$ and $H^6=\RR$. Also $Q^2=H^2/\la \omega \ra$, and $K^4=\ker (L_\omega: H^4\to \RR)$
are vector spaces of codimension one. We have the following:
 \begin{align*}
 H^0(M) &=\RR, \\
 H^1(M) &=0, \\
 H^2(M) &= Q^2 , \\
 H^3(M) &=H^3 \oplus K^2 x, \\
 H^4(M) &=Q^4 \oplus H^3 x, \\
 H^5(M) &=K^4 x, \\
 H^6(M) &=0, \\
 H^7(M) &=\la \omega^3 x\ra.
\end{align*}
The map $H^2(M)\x H^2(M)\to H^4(M)$ factors through $Q^2 \x Q^2 \to Q^4$. Hence if it is non-zero
then $Q^4\neq 0$. In particular, the Lefschetz map $L_\omega:H^2\to H^{4}$ is not an isomorphism,
so $B$ is not hard Lefschetz.

If $M$ admits a Sasakian structure, then there is a quasi-regular fibration $S^1 \to M \to B$ with
$B$ satisfying the hard Lefschetz property (it is a K\"ahler orbifold, so \cite{WZ} is applicable). This contradicts the above.
\end{proof}

Now we shall study the case of Sasakian $7$-manifolds in more detail. Let $M$ be a 
simply connected compact Sasakian $7$-dimensional manifold. Then 
 $$
 \SM=(H\otimes \bigwedge(x),d)
 $$
is a model for $M$, by Proposition \ref{prop:model}, where $H=H^*(B)$ is the cohomology
algebra of a simply connected compact $6$-dimensional K\"ahler orbifold. This algebra $H$
has a very rich structure: 
\begin{enumerate}
\item there is a canonical isomorphism $H^6 \cong \RR$, which is given by integration $\int_M:H^6\to \RR$;
\item $H$ is a Poincar\'e duality algebra, hence $H^3 \otimes H^3 \to \RR$ is an antisymmetric
bilinear pairing;
\item there is a scalar product on each $H^j$. This is given by the Hodge star operator
$* :H^j \to H^{6-j}$ combined with wedge and integration;
\item there is a distinguished element $\omega\in H^2$. This defines the space
of primitive forms $P=\la \omega\ra^\perp \subset H^2$. Hence $H^2=\la \omega\ra \oplus P$;
\item the Lefschetz map $L_\omega:H^2\to H^4$ is an isomorphism. Therefore
$H^4=\la \omega^2\ra \oplus \omega P$. On the other hand, for $\alpha\in P$ we have
$* \alpha=\alpha\wedge \omega$, and $* \omega = \frac12 \omega^2$. This implies
that  $\omega P=\la \omega^2\ra^{\perp}$, and $L_\omega: 
\la \omega\ra \oplus P \to \la \omega^2\ra \oplus \omega P$ is of the form
$L_\omega(\alpha)=L_\omega(\alpha_0+\alpha_1)=\frac12 *\alpha_0+*\alpha_1$,
where $\alpha=\alpha_0+\alpha_1$ is the decomposition according
to $H^2=\la \omega\ra \oplus P$.
\end{enumerate}

The Lefschetz map $L_\omega:H^2\to H^4$ is an isomorphism so there is an inverse
$L_\omega^{-1}:H^4\to H^2$. Using it, we can define a map 
$\SF: P \x P\x P\x P \to \RR$ by
 $$
 \SF(\alpha,\beta,\gamma,\delta)=\int_M L_\omega^{-1}(\alpha\wedge \beta) \wedge \gamma\wedge \delta .
 $$
This clearly factors through $\Sym^2P \x \Sym^2P$. Using (5) above, we have the alternative 
description
 $$
 \SF(\alpha,\beta,\gamma,\delta)= 2\la (\alpha\wedge \beta)_0, (\gamma\wedge \delta)_0 \ra +
\la (\alpha\wedge \beta)_1, (\gamma\wedge \delta)_1 \ra,
 $$ 
from where it follows that $\SF$ factors as a map\, $\Sym^2(\Sym^2P)\to \RR$.

Let $\SK_M$ be the kernel of the map $\Sym^2(\Sym^2 P)\to \Sym^4 P$. Then we define a map 
  $$
 \SF_M=\SF|_{\SK_M}: \SK_M\to \RR.
 $$
We have the following result.

\begin{theorem} \label{thm:formal-7-dim}
   Let $M$ be a simply connected compact Sasakian $7$-dimensional manifold.
 Then $M$ is formal if and only if $\SF_M=0$.
\end{theorem}

\begin{proof}
Using Theorem \ref{fm2:criterio2}, we only have to check whether $M$ is $3$-formal. For this we have to construct the minimal 
model $\rho:(\bigwedge V,d) \to \SM= (H\otimes \bigwedge(x),d)$ up to degree $3$. This is easy:
  \begin{align*}
  V^1 &=0, \\
 V^2 &= P , \\
 V^3 &= H^3 \oplus N^3, \qquad \text{where } N^3=\Sym^2P,
\end{align*}
where the differential is given by $d=0$ on $P$ and $H^3$, and
$d:N^3 \to \bigwedge V^2$ is the isomorphism $\Sym^2P \to \bigwedge^2 P$.
The map $\rho$ is given as follows.
$\rho: V^2 =P\to  \SM^2=H^2$ is defined as the obvious (inclusion) map, 
$\rho: H^3 \to \SM^3=H^3 \oplus H^2 x$ is the inclusion on the first summand,
and $\rho:N^3=\Sym^2P \to  \SM^3=H^3 \oplus H^2 x$ is defined as
$\rho(\alpha\cdot\beta)=L_{\omega}^{-1}(\alpha\wedge \beta) \, x$. Note that 
 $$
 d( \rho(\alpha\cdot\beta))=L_{\omega}^{-1}(\alpha\wedge \beta)  \, dx=
L_{\omega}^{-1}(\alpha\wedge \beta) \, \omega= \alpha\wedge \beta
 =\rho(\alpha)\wedge \rho(\beta)=\rho(\alpha\wedge \beta)=
\rho(d(\alpha\cdot\beta)),
$$
so $\rho$ is a DGA map. Clearly it is a $3$-equivalence (it induces an isomorphism
on cohomology up to degree $3$ and an inclusion on degree $4$).

The space of closed elements is $C^3=H^3$. Now let us check when the elements $z\in I(N^3)$ with $dz=0$ satisfy
$[\rho(z)]=0\in H^*(M)$. The only cases to check is when $z$ has degree $5$ or $7$.
If $z$ has degree $5$, then $[\rho(z)]\neq 0$ if and only if there exists
some $\beta\in P$, $[\rho(\beta)]\in H^2(M)$, such that $[\rho(z)]\wedge [\rho(\beta)]\neq 0$,
by Poincar\'e duality. Hence $[\rho(z\beta)]\neq 0$. This means that we can restrict 
to elements $z$ of degree $7$, that is $z \in N^3\cdot \bigwedge^2 P$.

Let $z \in N^3\cdot \bigwedge^2 P \cong \Sym^2P \x \Sym^2 P$. Then the
map $d:  N^3\cdot   \bigwedge^2P  \to  \bigwedge^4P$ coincides the full
symmetrization map $ \Sym^2P \x \Sym^2 P \to \Sym^4P$. So
 $$
 Z= \ker d|_{I(N^3)^7}=\SK_M \oplus \Ant^2(\Sym^2P),
 $$
where $\Ant^2(W)$ denotes the antisymmetric $2$-power of a vector space $W$.

Now we have to study the map 
 $$
 \rho: Z \to H^7(M) =H^6 x,
 $$
and see if this is non-zero. This is given (on the basis elements) by 
$$
 \rho ((\alpha\cdot\beta)\cdot (\gamma\cdot \delta))=
\left(L_{\omega}^{-1}(\alpha\wedge \beta)  \wedge \gamma\wedge\delta \right) x,
 $$
so $\SF_M=\rho|_{\SK_M}$. Note that $\rho$ automatically vanishes on $ \Ant^2(\Sym^2P)$,
hence $M$ is formal if and only if $\rho$ vanishes on $\SK_M$ if and only if $\SF_M=0$.

According to Theorem \ref{fm2:criterio2}, to check non-formality we have to test the relevant property (2)
on \emph{any} splitting $V^3=C^3+ N'{}^3$.
If we take another splitting $V^3=C^3+ N'{}^3$, then the projection $\pi:V^3\to N^3$ gives an
isomorphism $\pi:N'{}^3 \to N^3$, and so an isomorphism $N'{}^3\cdot \Sym^2P \cong N^3 \cdot \Sym^2P$.
Clearly, $d\circ \pi=d$ on $N'{}^3$, so the spaces of cycles correspond
$\SK' \cong \SK$. On the other hand $H^3 \cdot H^2\cdot H^2 =0$, so the maps 
$\rho:\SK\to H^6 x$ and $\rho:\SK'\to H^6 x$ also correspond. This means that 
the corresponding
$\SF$ and $\SF'$ coincide under the isomorphism $\SK\cong \SK'$. This means that
the choice of splitting is not relevant.
\end{proof}

This result means that the formality or non-formality of $M$ only depends on the cohomology algebra $H$.
Theorem \ref{thm:formal-7-dim} can be applied to the examples in Section 5.3 of
\cite{BFMT}. For instance for $B=\CP^1\x \CP^1\x \CP^1$, we have
a  simply connected Sasakian $7$-manifold which is non-formal (Theorem 12 of \cite{BFMT}). 
For $B=\CP^3$, we have obviously $P=0$ and hence $M$ is formal.

The element $\SF_M$ of Theorem \ref{thm:formal-7-dim}
is the {\it principal Massey product\/} defined by Crowley and Nordstr\"om in \cite{Johannes} 
for simply connected compact $7$-manifolds in general. The principal Massey product is the full obstruction
to formality for simply connected compact $7$-manifolds.

Now we deduce Corollary  \ref{cor:formal-7-dim}.

\begin{corollary}
 Let $M$ be a simply connected compact Sasakian $7$-dimensional manifold.
 Then $M$ is formal if and only if all triple Massey products are zero.
\end{corollary}

\begin{proof}
 Suppose that $\SF_M\neq 0$. We choose an orthonormal basis
for $H^2=\la e_0,e_1,\ldots, e_m\ra$, where $e_0=\frac{1}{\sqrt{3}} \omega$,
and $\la e_1,\ldots,e_m\ra =P$. The vector space $\SK_M$ is 
generated by elements of the form
 $$
 a_{ijkl}= (e_i \cdot e_j)\cdot (e_k\cdot e_l) - (e_k \cdot e_j)\cdot (e_i\cdot e_l),
 $$
for $1\leq i,j,k,l \leq m$ (here, as usual, the dot product means symmetric product).
Now define the numbers
 $$
 \lambda_{ijk}= \int_M e_i\wedge e_j\wedge e_k \in \RR,
 $$
for $1\leq i,j,k \leq m$. Note that these numbers are fully symmetric on $i,j,k$. 
Also $\lambda_{000}= \frac{2}{\sqrt{3}}$ and $\lambda_{ij0}=\frac{1}{\sqrt{3}}\delta_{ij}$, for $(i,j)\neq (0,0)$.
Then 
 $$
 L_\omega^{-1}(e_i\wedge e_j)=2*(e_i\wedge e_j)_0 +(e_i\wedge e_j)_1=2\lambda_{ij0}e_0 +\sum_{t>0} \lambda_{ijt}e_t.
 $$
So
 $$
\SF_M ( (e_i \cdot e_j)\cdot (e_k\cdot e_l))=2\lambda_{ij0}\lambda_{kl0} +\sum_{t>0}\lambda_{ijt}\lambda_{klt} \, .
 $$
Evaluating $\SF_M$ on $a_{ijkl}$ gives a set of equations to determine the formality of $M$.
$M$ is non-formal when there exists some $a_{ijkl}$ with $\SF_M(a_{ijkl})\neq 0$.
By \cite{Johannes}, we have that the triple Massey product $\la e_i, e_j, e_k\ra$ is a well-defined element of $H^5(M)$ and
it satisfies
 $$
 \SF_M(a_{ijkl}) =\la e_i, e_j, e_k\ra \cup e_l\, .
 $$
So $\la e_i, e_j, e_k\ra\neq 0$, as required.
\end{proof}

This result is of relevance since it is not known if   
for general simply connected compact $7$-dimensional manifolds  there
are obstructions to formality different from triple Massey products,
as remarked in \cite{Johannes}. It is true that for higher
dimensional manifolds, there are obstructions to formality
even when all Massey products (triple and higher order) can be zero.

\end{document}